\documentclass[12pt]{article}

\usepackage[margin=1in]{geometry}
\usepackage{graphicx}

\usepackage{url}
\usepackage{hyperref}
\usepackage{comment}
\usepackage{amsthm}
\usepackage{amsmath}
\usepackage{amssymb}

\usepackage{mathtools}
\usepackage{enumitem}
\usepackage{cleveref}
\usepackage{thmtools}

\usepackage{xcolor}

\usepackage[sorting=nyt,sortcites=true,maxbibnames=9,autopunct=true,autolang=hyphen,hyperref=true,abbreviate=false, backend=biber]{biblatex}
\usepackage{mathrsfs}
\usepackage{tikz-cd}
\usepackage{cleveref}

\usepackage{amsthm}
\usepackage{amsfonts}
\usepackage{amsmath}
\usepackage{mathrsfs}
\usepackage{graphics}
\usepackage{graphicx}
\usepackage{hyperref}
\usepackage{comment}
\usepackage{color}

\usepackage{algorithm} 
\usepackage{algpseudocode}

\def\1{{\sf 1}}

\def\II{{\sf I}}
\def\k{{\sf k}}
\def\u{{\sf u}}

\def\x{{\sf x}}\def\y{{\sf y}}
\def\w{{\sf w}}\def\z{{\sf z}}
\DeclareMathOperator{\tr}{tr}
\DeclareMathOperator{\spec}{Spec}
\DeclareMathOperator{\vol}{vol} 
\DeclareMathOperator{\ncut}{NCut}
\DeclareMathOperator{\cut}{cut}

\newtheorem{proposition}{Proposition}[section]
\newtheorem{lemma}[proposition]{Lemma}
\newtheorem{theorem}[proposition]{Theorem}

\newtheorem{corollary}[proposition]{Corollary}
\newtheorem{example}[proposition]{Example}
\newtheorem{definition}[proposition]{Definition}
\newtheorem{obs}[proposition]{Remark}

\def\2{{\sf 1}}

\usepackage{hyperref}

\begin{document}
\newcommand{\Sparsify}{\sf{Sparsify}\kern1pt}

\title{Kemeny’s constant via matrix compression\\ and eigenvalue interlacing}



%

\author{
A. Abiad
\thanks{\texttt{a.abiad.monge@tue.nl}, Department of Mathematics and Computer Science, Eindhoven University of Technology, The Netherlands. Department of Mathematics and Data Science of Vrije Universiteit Brussel, Belgium}
\qquad
\'A. Carmona\thanks{\texttt{angeles.carmona@upc.edu}, Departament de Matem\`atiques, Universitat Polit\`ecnica de Catalunya, Spain}
\and
A. M. Encinas\thanks{\texttt{andres.marcos.encinas@upc.edu}, Departament de Matem\`atiques, Universitat Polit\`ecnica de Catalunya, Spain}
\and
E. Ghorbani\thanks{\texttt{ebrahim.ghorbani@tuhh.de}, Institute for Algorithms and Complexity, Hamburg University of Technology, Germany}
\and
M.J. Jim\'enez\thanks{\texttt{maria.jose.jimenez@upc.edu}, Departament de Matem\`atiques, Universitat Polit\`ecnica de Catalunya, Spain} 
\and
\'A. Samperio
\thanks{\texttt{alvaro.samperio@uva.es}, Instituto de Investigaci\'on en Matem\'aticas, Universidad de Valladolid, Spain} 
}

\date{}

\maketitle

\begin{abstract}
Kemeny's constant quantifies the expected time for a random walk to reach a randomly chosen vertex, capturing global properties of a Markov chain. We develop a matrix-analytic framework for bounding Kemeny's constant of a finite connected weighted graph using degree-weighted compressions of the normalized adjacency matrix, pinching inequalities, and eigenvalue interlacing. Our main partition theorem gives lower bounds in terms of the compressed spectrum, with complete equality characterizations, and converts structural graph information into spectrally computable estimates. For instance, when applied to proper color partitions, the proposed method yields a sharp lower bound on Kemeny's constant in terms of the chromatic number
\[
    K(G)\ge n-2+\frac{1}{\chi(G)},
\]
which extends a bipartite bound of Ciardo, Dahl, and Kirkland (2022) to arbitrary chromatic number, and which is incomparable with the normalized Hoffman bound by Chung (1997). For unweighted graphs, we also characterize all equality cases of this bound. We further illustrate the power of the proposed matrix framework by deriving spectral bounds for NP-hard graph problems involving normalized cuts and conductance. Our results include an asymptotically sharp conductance bound, two-sided interlacing bounds from principal submatrices and quotient matrices, and estimates for the change in Kemeny's constant under connectivity-preserving deletion of multiple edges.\\



\end{abstract}


\section{Introduction}

Kemeny's constant $K(G)$ is a fundamental parameter in the theory of random walks on graphs, quantifying the expected number of steps required for a random walker to reach a target vertex sampled from the stationary distribution \cite{KeSn60,Lo16}. Originally defined by Kemeny and Snell for finite Markov chains \cite{KeSn60}, $K(G)$ serves as a global indicator of graph connectivity and mixing efficiency \cite{Lo16}. Since its introduction, Kemeny's constant has a wide range of applications, including modeling the spread of infectious diseases (predicting how rapidly an outbreak will reach epidemic levels, see e.g. \cite{Yetal2020,Detal2023}), molecular conformation dynamics (identifying the presence or absence of metastable sets, see e.g. \cite{MMLR2024}), and urban road networks (assessing how well connected a network is, see e.g. \cite{CKS2011}).

Understanding how structural graph properties and sub-configurations affect the behavior of $K(G)$ has recently received a considerable amount of attention, see e.g.\cite{BK2019,APY2021,derivative,ABCMP2023,BCK2022,FKK2022,KDP2024,KLMZ2024}. Recent work has explored the impact of structural changes on Kemeny's constant, ranging from local vertex and edge removals to graph compressions, bridge structural formulas, and network centralities. From a matrix-analytic perspective, Breen and Kirkland \cite{BK2019} studied the sensitivity of Kemeny's constant to perturbations of the transition matrix through a structured condition number, while Altafini et al.\ \cite{ABCMP2023} investigated variations of Kemeny's constant under edge removal and used them to define an edge centrality measure. These previous results illustrate the usefulness of matrix methods in understanding the sensitivity of Kemeny's constant to structural and probabilistic perturbations.

In this paper we develop an algebraic matrix-based approach to obtain sharp bounds on $K(G)$ from vertex partitions, edge deletions, and graph substructures. The central idea of the proposed approach is to exploit the spectral formulation of Kemeny's constant by projecting the normalized adjacency matrix of $G$ onto degree-weighted quotient spaces. This compression separates the spectral information associated with the chosen partition from that contained in its orthogonal complement, allowing structural properties of the graph to be translated into bounds on $K(G)$. In particular, whenever a graph parameter is characterized through a suitable vertex partition or graph substructure, our approach provides a way to relate that parameter to the spectrum and hence to Kemeny's constant. Thus, structural descriptions of graph parameters can be converted into spectrally computable bounds, even when the underlying parameter is difficult to determine exactly (here we will estimate several NP-hard parameters for an illustration). The matrix viewpoint is central: $K(G)$ can be written as a trace involving the normalized Laplacian, and our bounds follow from matrix compression, trace inequalities, and eigenvalue interlacing. In contrast with existing perturbation-based approaches, our method uses compressions associated with vertex partitions and substructures to extract spectral information directly from the graph structure.

For instance, when applied to proper color partitions, where each part forms an independent set, the proposed matrix compression strategy yields a sharp lower bound on Kemeny's constant governed by the chromatic number $\chi(G)$
\[
K(G) \ge n - 2 + \frac{1}{\chi(G)},
\]
which in turn extends a bound of Ciardo, Dahl, and Kirkland for bipartite graphs \cite{CiDaKi20}. We provide a complete characterization of the equality cases, proving that equality holds if and only if $G$ is a complete bipartite graph or a balanced complete $r$-partite graph. Equivalently, Kemeny's constant yields a new spectral lower bound on the chromatic number, a well known NP-hard graph parameter. Furthermore, a direct comparison reveals that our chromatic number bound and the classical normalized Hoffman bound on the chromatic number (see \cite[Theorem~6.7]{Chung1997} and \cite{ElphickWocjan2015}) are in general incomparable. Hence, our estimate provides an alternative sharp spectral bound on $\chi(G)$ rather than a refinement of the normalized Hoffman bound. The same projection method extends directly to partitions defined by graph separators, providing explicit lower bounds on $K(G)$ via the normalized cut introduced in \cite{ShiMalik2000} and the graph conductance.

In the final section of the paper, we use several eigenvalue interlacing techniques to estimate structural perturbations of Kemeny's constant. By applying Cauchy interlacing and quotient matrix interlacing to the normalized adjacency matrix, we establish tight two-sided bounds on $K(G)$ derived from principal submatrices and quotient partitions, yielding a sharp upper bound parameterized by local edge degrees. Moreover, by using the normalized Laplacian edge-interlacing technique of Hall, Patel, and Stewart \cite{edgeinterlacing}, we extend the analysis from single-edge deletions to multi-edge removals. This provides spectral bounds on the variation of Kemeny's constant under connectivity-preserving edge deletion, linking graph conductance directly to the structural sensitivity of $K(G)$.

\section{Preliminaries}

 In this paper, $G=(V,E,c)$ is a finite, simple, connected, undirected weighted graph, where $c:V\times V\longrightarrow[0,\infty)$ assigns a positive symmetric weight $c(i,j)=c(j,i)=c_{ij}>0$ to each edge $\{i,j\}\in E$ and $c_{ij}=0$ otherwise. The unweighted graph $(V,E)$ is called the \emph{underlying graph} of $G$. We assume that $|V|=n\ge2$; in particular, there are no isolated vertices or loops. Unless explicitly stated otherwise, all results below are formulated for weighted graphs, including the compression, chromatic, cut, interlacing, and edge-weight reduction results. Graph parameters such as the chromatic number $\chi(G)$, independence number $\alpha(G)$, color classes, and cliques refer to the underlying graph. Results or examples that require unit edge weights are identified explicitly.
Recall that the \emph{degree} of vertex $i$ is $$k_i=\sum\limits_{j=1}^nc_{ij},$$
and the \emph{volume} of $G$ is
$${\rm vol}(G)=\displaystyle\sum\limits_{i=1}^nk_i,$$
and in the unweighted case ${\rm vol}(G)=2|E|.$ 
For $S\subseteq V$, we write
${\rm vol}(S)=\sum_{i\in S}k_i$
and
$${\rm cut}(S,\bar S)=\sum_{u\in S,v\in\bar S}c_{uv}.$$
Thus, volumes and cuts are defined using edge weights; in the unweighted case, the cut is simply the number of edges between $S$ and $\bar S$.

The \emph{adjacency matrix} $A$ of a graph $G$ is the $n\times n$ symmetric matrix with elements $c_{ij}$.  We denote by ${\sf e}^i$ the $i$-th vector of the standard basis. Throughout, vectors are represented using {\it sans serif} style. For instance, $\k$ is the vector of degrees of vertices of $G$. A graph is called {\it regular} if the degree vector is constant. Additionally, for any real number $\alpha$, we define $\alpha^{\#}$ as $\dfrac{1}{\alpha}$ when $\alpha\not=0,$ and $0$ otherwise.  For a square matrix $N$ whose group inverse exists,   denote its \emph{group inverse} by $N^\#$.  For a real symmetric matrix it always exists and coincides with the Moore--Penrose inverse.

The {\it Laplacian matrix} $L_G$ of $G$ is the $n\times n$ symmetric matrix given by $L_G=D-A$, where $D=\text{diag}(k_1,\ldots,k_n)$ is the diagonal degree matrix. The \emph{normalized Laplacian matrix} of a graph, denoted ${\cal{L}}_G$, is defined as
$${\cal{L}}_G=D^{-1/2}L_GD^{-1/2}=I-D^{-1/2}AD^{-1/2},$$
where $M_G=D^{-\frac{1}{2}}AD^{-\frac{1}{2}}$ is called {\it normalized adjacency matrix}.

A random walk on $G$ induces a Markov chain with \emph{transition probability matrix} $P$, defined as $P=D^{-1}A$. For a finite, irreducible Markov chain, the transition probability matrix 
 $P$, its steady state probability vector $\pi$, and the all-ones vector $\sf 1$ satisfy $P{\sf 1}={\sf 1}$ and $\pi^\top P=\pi^\top$. It is known that $$\pi_i=\dfrac{k_i}{{\rm vol}(G)}.$$ 
 
Kemeny's constant is a fundamental parameter in the study of Markov chains and random walks on graphs. It quantifies the expected number of steps a random walker takes to reach a randomly chosen vertex, regardless of the starting position, under the stationary distribution. Intuitively, Kemeny's constant provides a global measure of how efficiently information or influence propagates through a graph. It is particularly useful in assessing the connectivity and mixing properties of a graph. Kemeny and Snell \cite{KeSn60} established a direct connection between Kemeny's constant and random walks
\begin{equation}\label{eq:pim} \tilde{K}(G)=\sum_{i=1}^{n} \pi_i m_{ji}, \end{equation}
where $m_{ji}$ denotes the \emph{mean first passage time} (the expected number of steps required for a random walk to reach vertex 
$i$, starting from vertex $j$). Notably, this sum is constant regardless of the starting vertex $j$.

Some authors define Kemeny's constant as $K(G)=\tilde{K}(G)-1$. We will use the value $K(G)$ in this work.

It is known that the Kemeny constant can be expressed by means of the eigenvalues of several matrices associated with a graph. In particular, there is a way to compute Kemeny's constant by using the eigenvalues of the
probability transition matrix. First, observe that since $P$ is similar to the  normalized adjacency matrix, they have the same eigenvalues. Because the normalized adjacency matrix is symmetric all its eigenvalues, and so the eigenvalues of $P$, are real.

 We denote the eigenvalues of the normalized adjacency matrix as 
 $$1=\lambda_1 > \lambda_2\ge\cdots \geq \lambda_n\ge -1.$$
 The eigenvalues of the normalized Laplacian matrix ${\cal{L}}_G$  are denoted by 
 $$2\ge\mu_1\ge\cdots\ge\mu_{n-1}>\mu_n=0,$$
 and they satisfy the relationship $\mu_i = 1 - \lambda_{n-i+1}$, for $i = 1, \ldots, n$. 

 It is well-known that Kemeny's constant can be computed in terms of eigenvalues of the normalized Laplacian as
 \begin{equation}\label{Lovasz}
  K(G)=\sum_{j=1}^{n-1}\frac{1}{\mu_j}=\sum_{j=2}^{n}\frac{1}{1-\lambda_j}\end{equation}
	and hence $K(G)={\rm tr}({\cal L}_G^\#)$; 
  see \cite[Eq. (3.3)]{Lo16}. Furthermore, Wang, et al. \cite[Theorem 4]{Wa17} showed that Kemeny's constant can also be expressed in terms of the group inverse of the combinatorial Laplacian, specifically, 

\begin{equation*}  \label{thm:wangetal}
K(G)={\rm tr}(L_G^\# D)-\dfrac{1}{\rm{vol}(G)}{\k}^{\sf T}L_G^{\#}{{\sf k}}.
\end{equation*}
 The same identity was  obtained in \cite{CaJiMa23} using a different approach and later generalized for the case of Schr\"odinger random walks in \cite{CaEnJiMa24}.

\section{Vertex partitions and spectral compression}\label{sec:vertex-partitioning}

In this section, we study Kemeny’s constant through the spectral information carried by partitions of the vertex set. Given a partition of $V(G)$, we compress the normalized adjacency matrix to the subspace generated by the parts. The resulting degree-weighted quotient matrix describes how the parts interact, while the complementary subspace contains the remaining spectral information.

We combine this decomposition with a pinching inequality to control the contribution of the two subspaces to Kemeny’s constant. This leads to our main partition theorem, including a characterization of the equality case. A closed form of the bound is then obtained by applying Cauchy--Schwarz to the eigenvalues of the quotient matrix.

In the next section, we apply this result to several natural partitions, including color partitions and partitions associated with normalized cut and conductance.

\subsection{A pinching inequality}
We use the notation $A\succ0$ and $A\succeq0$ to indicate that the matrix $A$ is positive definite and positive semidefinite, respectively.  More generally, $A\succeq B$ means that $A-B\succeq0$.

A \emph{pinching map} is a linear map that, with respect to a fixed block decomposition, deletes all off-diagonal blocks and keeps only the diagonal ones. For example, the map
$$\begin{pmatrix}
	A_{11}&A_{12}\\
	A_{21}&A_{22}
\end{pmatrix}
\longmapsto
\begin{pmatrix}
	A_{11}&0\\
	0&A_{22}
\end{pmatrix}$$
is a pinching map.
More explicitly, if $P_1$ and $P_2$ are the orthogonal projections onto the two block-coordinate subspaces, then the corresponding pinching map is
$$\Phi(A)=P_1AP_1+P_2AP_2.$$
Since the function $f(t)=t^{-1}$ is \emph{operator convex} on $(0,\infty)$, Davis's pinching inequality \cite{Davis1957}, or equivalently Jensen's operator inequality \cite{HansenPedersen2003}, gives
$$\Phi(A^{-1})\succeq\Phi(A)^{-1}$$
for every positive definite matrix $A$.

For this inverse-function pinching inequality, equality holds precisely when $A$ commutes with the orthogonal projections defining the pinching map, as is proved directly below. In the two-block decomposition
$$
A=
\begin{pmatrix}
	A_{11}&X\\
	X^\top&A_{22}
\end{pmatrix},
$$
this is equivalent to $X=0$.

To avoid introducing the language and machinery of operator theory, we give below a self-contained proof of the particular pinching inequality needed here, including the characterization of equality.

We shall use the following two standard facts concerning positive
definite matrices and Schur complements; see Theorem~7.7.7 and Corollary~7.7.4(a) in \cite{HornJohnson}.

\begin{itemize}
	\item[(S1)] If
	$N=	\begin{pmatrix}
		P&R\\
		R^\top&Q\end{pmatrix}$	is symmetric and $P\succ0$, then $N\succ0$ if and only if its Schur
	complement	$Q-R^\top P^{-1}R$ is positive definite.
	\item[(S2)] If $B\succeq A\succ0$, then	$A^{-1}\succeq B^{-1}$.
\end{itemize}

\begin{lemma}\label{lem:pinching}
	Let	$$N=\begin{pmatrix} N_1 & X\\ X^\top & N_2\end{pmatrix}$$
	be real symmetric with $I-N$ positive definite. Then
	$$\tr (I-N)^{-1}\ \ge\ \tr (I-N_1)^{-1}+\tr (I-N_2)^{-1},$$
	with equality if and only if $X=0$.
\end{lemma}

\begin{proof}
		Let $A = I - N$. We  express $A$ as the following partitioned matrix:
		$$A = \begin{pmatrix} I-N_1 & -X \\ -X^\top & I-N_2 \end{pmatrix} = \begin{pmatrix} A_{11} & A_{12} \\ A_{12}^\top & A_{22} \end{pmatrix}.$$
		We have $I - N$  positive definite, i.e. $A \succ 0$, so its principal submatrices are also positive definite. Thus,  
        $A_{11} \succ 0$ and $A_{22} \succ 0$.

		The inverse of the symmetric block matrix $A$ can be expressed using its Schur complements (see \cite[p.~18]{HornJohnson}). The diagonal blocks of $A^{-1}$ are given by:
		$$
		(A^{-1})_{11} = \left(A_{11} - A_{12} A_{22}^{-1} A_{12}^\top\right)^{-1},\qquad 
        (A^{-1})_{22}=\left(A_{22} - A_{12}^\top A_{11}^{-1} A_{12}\right)^{-1}.		$$
		Taking the trace we obtain:
		$$\tr (A^{-1}) = \tr \left((A_{11} - A_{12} A_{22}^{-1} A_{12}^\top)^{-1}\right) + \tr \left((A_{22} - A_{12}^\top A_{11}^{-1} A_{12})^{-1}\right).$$
	Since $A_{22} \succ 0$, its inverse is also positive definite. Consequently, the congruent matrix term $A_{12} A_{22}^{-1} A_{12}^\top$ is positive semi-definite.
		Subtracting this from $A_{11}$ results in
		$A_{11} - A_{12} A_{22}^{-1} A_{12}^\top \preceq A_{11}$.
	By (S2), we obtain:
		$(A_{11} - A_{12} A_{22}^{-1} A_{12}^\top)^{-1} \succeq A_{11}^{-1}$
		and thus
		$$\tr \left((A_{11} - A_{12} A_{22}^{-1} A_{12}^\top)^{-1}\right) \ge \tr (A_{11}^{-1}).$$
	Similarly
		$$\tr \left((A_{22} - A_{12}^\top A_{11}^{-1} A_{12})^{-1}\right) \ge \tr (A_{22}^{-1}).$$
		Summing these two inequalities results in
		$$\tr (A^{-1}) \ge \tr (A_{11}^{-1}) + \tr (A_{22}^{-1}).$$

    For equality to hold, the individual trace inequalities must hold with equality
		\begin{equation}\label{eq:trA11}
			\tr \left((A_{11} - A_{12} A_{22}^{-1} A_{12}^\top)^{-1}\right) - \tr (A_{11}^{-1}) = 0
		\end{equation}
		Because the difference matrix $(A_{11} - A_{12} A_{22}^{-1} A_{12}^\top)^{-1} - A_{11}^{-1}$ is positive semi-definite, its trace can only be zero if the matrix itself is the zero matrix. Thus \eqref{eq:trA11} is equivalent with $A_{11} - A_{12} A_{22}^{-1} A_{12}^\top = A_{11}$
		which simplifies to $A_{12} A_{22}^{-1} A_{12}^\top = 0$.
		Since $A_{22}^{-1} \succ 0$, this quadratic form equals zero if and only if $A_{12} = 0$. Given that $A_{12} = -X$, the equality holds if and only if $X = 0$.
\end{proof}

\subsection{Subspace compressions}

Orthogonal compression is a fundamental technique in spectral graph
theory. If $N$ is a real symmetric matrix and $S$ is a matrix with
orthonormal columns, then
$B=S^\top NS$
is called the \emph{compression} of $N$ to the column space of $S$.
Quotient matrix interlacing (Corollary~\ref{cor:2.3Haemerspaper}) is a
special case of this construction. 



\begin{definition}
\label{def:weighted-quotient}
Let $\mathcal P=\{V_1,\dots,V_m\}$ be a partition of $V(G)$ into nonempty parts.
The \emph{degree-weighted quotient matrix} of $\mathcal P$ is the $m\times m$
symmetric matrix $B=B(\mathcal P)$ with entries
$$B_{ij}={\sf s}_i^{\top}M{\sf s}_j,\qquad\text{where}\qquad
{\sf s}_i=\frac{1}{\sqrt{\vol(V_i)}}\,D^{1/2}\1_{V_i},$$
that is,
$$B_{ii}=\frac{2e(V_i)}{\vol(V_i)},\qquad
B_{ij}=\frac{e(V_i,V_j)}{\sqrt{\vol(V_i)\vol(V_j)}}
\quad(i\ne j),$$
where 
$e(V_i,V_j)=\sum_{u\in V_i,\,v\in V_j}c_{uv}$ and $e(V_i)=\sum_{\{u,v\}\subseteq V_i}c_{uv}$ are total edge weights (edge counts in the unweighted case). Here and below $M=M_G$.
We write
$$\tau(\mathcal P)\ \coloneqq\ \tr B(\mathcal P)
=\sum_{i=1}^{m}\frac{2e(V_i)}{\vol(V_i)},$$
which measures how much of the volume of $G$ is absorbed inside
the parts. In particular $\tau(\mathcal P)=0$ if and only if every part is an
independent set, that is, if and only if $\mathcal P$ is the color partition of
a proper coloring of $G$.
\end{definition}

 We recall that if $a_1,\ldots,a_p$ are positive reals, then the Cauchy--Schwarz inequality implies
 $\bigl(\sum_{i=1}^p\frac{1}{a_i}\bigr)\bigl(\sum_{i=1}^pa_i\bigr)\ge p^2$, or equivalently
 \[\sum_{i=1}^p\frac{1}{a_i}\ge \frac{p^2}{\sum_{i=1}^pa_i}.\]
 Equality holds if and only if $a_1=\cdots=a_p$.
 We will use both the inequality and this equality condition.

\begin{theorem}[Partition bound]\label{thm:partition-pinching}
Let $G$ be a connected weighted graph of order $n\ge3$ and let
$\mathcal P=\{V_1,\dots,V_m\}$ be a partition of $V(G)$ into $2\leq m\le n-1$ nonempty
parts. Let $1=\theta_1>\theta_2\ge\cdots\ge\theta_m$ be the eigenvalues of $B=B(\mathcal P)$ and put $\tau=\tau(\mathcal P)$ and $M=M_G$.
Then
\begin{equation}\label{eq:partition-sharp}
K(G) \ge \sum_{i=2}^{m}\frac{1}{1-\theta_i}+\frac{(n-m)^2}{n-m+\tau},
\end{equation}
with equality if and only if $\spec(M)=\spec(B)\uplus\{\rho^{(n-m)}\}$,  where $\rho=-\tau/(n-m)$.
\end{theorem}
\begin{proof}
Since $G$ is connected, $\vol(V_i)>0$ for each $i$, so the vectors
${\sf s}_1,\dots,{\sf s}_m$ are well defined; they form an orthonormal system
because the $V_i$ are pairwise disjoint and each ${\sf s}_i$ is a unit vector.
Let $U=\operatorname{span}\{{\sf s}_1,\dots,{\sf s}_m\}$ and put
$$\u=\frac{1}{\sqrt{\vol(G)}}\,D^{1/2}\1 .$$
As $A\1=D\1$ we have $M\u=\u$, and $\|\u\|=1$, so $\u$ is a unit eigenvector of
$\mathcal L_G=I-M$ for the eigenvalue $\mu_n=0$. Moreover,
\begin{equation}\label{eq:u-in-U}
\u=\sum_{i=1}^{m}c_i{\sf s}_i,\qquad
c_i=\sqrt{\vol(V_i)/\vol(G)},
\end{equation}
so $\u\in U$. Let $R$ be an $n\times(n-1)$ matrix whose columns form an orthonormal basis of
$\u^{\perp}$. Then $R^{\top}\mathcal L_GR$ has spectrum $\mu_1,\dots,\mu_{n-1}$, all
positive, hence is positive definite, and
\begin{equation}\label{eq:kemeny-reduced-laplacian}
K(G)=\tr \!\big((R^{\top}\mathcal L_GR)^{-1}\big).
\end{equation}
Set $\mathsf c=(c_1,\dots,c_m)^{\top}$. Using $M\u=\u$, \eqref{eq:u-in-U} and
the orthonormality of the ${\sf s}_i$,
$$(B\mathsf c)_i=\sum_{j=1}^{m}({\sf s}_i^{\top}M{\sf s}_j)c_j
={\sf s}_i^{\top}M\Big(\sum_{j=1}^{m}c_j{\sf s}_j\Big)
={\sf s}_i^{\top}M\u={\sf s}_i^{\top}\u=c_i,$$
so $B\mathsf c=\mathsf c$ and $1$ is an eigenvalue of $B$.
Writing ${\sf S}=({\sf s}_1\ \cdots\ {\sf s}_m)$, for any $\y\in\mathbb R^m$ we have
$$\y^{\top}B\y=({\sf S}\y)^{\top}M({\sf S}\y)\le\|\y\|^2.$$
Equality requires ${\sf S}\y$ to be a multiple of $\u$, since the eigenvalue $1$ of $M$ is simple. Hence $1$ is also the simple largest eigenvalue of $B$, and $\theta_i<1$ for $i\ge2$.
As $\tr B=\tau$,
\begin{equation}\label{eq:sum-theta}
\theta_2+\cdots+\theta_m=\tau-1 .
\end{equation}

Decompose $\u^{\perp}=U_0\oplus W$, where $U_0=U\cap\u^{\perp}$ has dimension
$m-1$ and $W=U^{\perp}$ has dimension $n-m$. Choosing
\begin{equation}\label{eq:Def-R}
R=\begin{pmatrix}R_{U_0}&R_W\end{pmatrix},
\end{equation}
whose blocks are orthonormal bases of $U_0$ and $W$, we obtain
\begin{equation}\label{eq:RtMR}
R^{\top}MR=\begin{pmatrix}B_0&X\\X^{\top}&C\end{pmatrix}.
\end{equation}
For every $\w\in U_0$ we have $\u^{\top}M\w=(M\u)^{\top}\w=\u^{\top}\w=0$, so
the compression of $M$ to $U$ decomposes, in an appropriate basis for
$U=\operatorname{span}(\u)\oplus U_0$, as the direct sum of $[1]$ and
$B_0=R_{U_0}^{\top}MR_{U_0}$. Hence the eigenvalues of $B_0$ are exactly
$\theta_2,\dots,\theta_m$, and $\tr(R^{\top}MR)=\tr M-\u^{\top}M\u=-1$ gives, with
	\eqref{eq:sum-theta},
	\begin{equation}\label{eq:trC}
		\tr C=-\tau .
	\end{equation}
Since $\mathcal L_G=I-M$,
$$R^{\top}\mathcal L_GR=I-R^{\top}MR
=\begin{pmatrix}\II-B_0&-X\\-X^{\top}&\II-C\end{pmatrix}\succ0 .$$
Applying \Cref{lem:pinching} to $R^{\top}MR$ and using
\eqref{eq:kemeny-reduced-laplacian},
\begin{equation}\label{eq:pinch}
K(G) \ge \tr (\II-B_0)^{-1}+\tr (\II-C)^{-1}
=\sum_{i=2}^{m}\frac{1}{1-\theta_i}+\tr (\II-C)^{-1},
\end{equation}
with equality if and only if $X=0$.
 As $\II-C\succ0$, Cauchy--Schwarz and
	\eqref{eq:trC} give
	\begin{equation}\label{eq:CS-C}
		\tr (\II-C)^{-1}\ \ge\
		\frac{(n-m)^2}{\tr (\II-C)}=\frac{(n-m)^2}{n-m+\tau},
	\end{equation}
	with equality if and only if $C=\rho\,\II$, where $\rho=-\tau/(n-m)$. This proves
	\eqref{eq:partition-sharp}.
	
	If equality holds throughout then $X=0$ and $C=\rho\II$, so
	$$\spec(M)=\{1\}\uplus\spec(B_0)\uplus\spec(C)
	=\spec(B)\uplus\{\rho^{(n-m)}\}.$$
    Conversely, assume that identity;
	then $\spec(R^{\top}MR)=\spec(B_0)\uplus\{\rho^{(n-m)}\}$,
	and comparing sums of squared eigenvalues,
	$$\tr (B_0^2)+2\tr (X^{\top}X)+\tr (C^2)
	=\tr (B_0^2)+(n-m)\rho^{2}.$$
	By \eqref{eq:trC} and Cauchy--Schwarz,
	$\tr (C^2)\ge\tau^{2}/(n-m)=(n-m)\rho^{2}$, so
	$\tr (X^{\top}X)\le0$. Hence $X=0$ and
	$\tr (C^2)=(n-m)\rho^{2}$, i.e.\ $C=\rho\II$, and equality holds in
	both \eqref{eq:pinch} and \eqref{eq:CS-C}.
\end{proof}

\begin{theorem}\label{thm:partition-closed}
	With the notation of \Cref{thm:partition-pinching}, we have that
	\begin{equation}\label{eq:partition-closed}
		K(G)\ \ge\ \frac{(m-1)^2}{m-\tau}+\frac{(n-m)^2}{n-m+\tau},
	\end{equation}
	with equality if and only if
	$\spec(M)=\bigl\{1,\big(\tfrac{\tau-1}{m-1}\big)^{(m-1)},\rho^{(n-m)}\bigr\}$.
\end{theorem}

\begin{proof}
	Since $\II-B_0\succ0$ we have $\theta_i<1$ for $i\ge2$, so $\tau<m$ by
	\eqref{eq:sum-theta}, and Cauchy--Schwarz gives
	$$\sum_{i=2}^{m}\frac{1}{1-\theta_i}\ \ge\
	\frac{(m-1)^2}{\sum_{i=2}^{m}(1-\theta_i)}=\frac{(m-1)^2}{m-\tau},$$
	with equality if and only if $\theta_2=\cdots=\theta_m$; combining with
	\Cref{thm:partition-pinching} gives \eqref{eq:partition-closed}. Equality forces
	equality in \Cref{thm:partition-pinching} together with
	$\theta_2=\cdots=\theta_m=\frac{\tau-1}{m-1}$, which is the stated spectral
	identity. Conversely, that identity gives directly
	$$K(G)=\frac{m-1}{1-\frac{\tau-1}{m-1}}+\frac{n-m}{1-\rho}
	=\frac{(m-1)^2}{m-\tau}+\frac{(n-m)^2}{n-m+\tau}.$$
\end{proof}
\begin{obs}\label{rem:partition-endpoint}
The restriction $m\le n-1$ avoids the undefined scalar $\rho=-\tau/(n-m)$. For the singleton partition $m=n$, one has $B=M$ and $\tau=0$, so the quotient spectral sum is exactly $K(G)$. Separately, Cauchy--Schwarz and $\sum_{j=1}^{n-1}\mu_j=n$ give
$$K(G)\ge\frac{(n-1)^2}{n}.$$
\end{obs}
\section{Chromatic number, normalized cut, and conductance}

We now apply the partition bounds from \Cref{sec:vertex-partitioning} to
natural graph partitions. Choosing a color partition gives a lower bound on
the chromatic number, while partitions arising from normalized cut and
conductance give bounds in terms of standard measures of how well a graph can
be separated into weakly connected parts.

These applications illustrate how the parameter $\tau(\mathcal P)$ connects
Kemeny's constant with different notions of graph partitioning. In each case,
a suitable partition inserted into \Cref{thm:partition-pinching} or
\Cref{thm:partition-closed} yields a lower bound on $K(G)$. For coloring, we
also characterize the equality cases and compare the resulting bound with
the normalized Hoffman bound.

The same viewpoint will be used throughout this section: a suitable
partition separates the spectral information into a quotient part and its
orthogonal complement, allowing us to relate $K(G)$ to structural parameters
of the graph.

\subsection{Graph coloring}

The bound below holds for weighted graphs; the equality classification is stated for the unweighted case. We apply the partition bound to a proper coloring of the underlying graph. Since each color
class is an independent set, we have $\tau(\mathcal P)=0$. This yields a lower
bound on $K(G)$ in terms of the chromatic number, together with a
characterization of the equality cases.

For connected unweighted bipartite graphs, Ciardo, Dahl and Kirkland proved the following
sharp bound.

\begin{theorem}[Ciardo, Dahl and Kirkland {\cite[Proposition~4.1]{CiDaKi20}}]
\label{thm:CDK}
Let $G$ be a connected unweighted bipartite graph of order $n$. Then $K(G)\ge n-\frac32$,
with equality if and only if $G$ is complete bipartite.
\end{theorem}

Our partition bound extends this result from bipartite graphs to graphs of
arbitrary chromatic number. Indeed, applying
\Cref{thm:partition-pinching} to a color partition immediately gives the
following bound.

\begin{corollary}    
\label{cor:kemeny-chromatic}
	Let $G$ be a connected weighted graph of order $n\ge2$. Then
	$$K(G)\ \ge\ n-2+\frac{1}{\chi(G)} .$$
	If $G$ is unweighted, equality holds if and only if one of the following holds:
	\begin{itemize}
		\item[(i)] $\chi(G)=2$ and $G$ is complete bipartite;
		\item[(ii)] $\chi(G)=r\ge3$, $r$  divides $n$, and $G\cong T_r(n)$, the
		balanced complete $r$-partite graph.
	\end{itemize}
\end{corollary}
\begin{proof}
Let $r=\chi(G)$ and let $\mathcal P=\{V_1,\dots,V_r\}$ be a color partition of
the underlying graph. If $r=n$,  then \Cref{rem:partition-endpoint} gives $K(G)\ge(n-1)^2/n=n-2+1/n$. If $G$ is unweighted, then $G=K_n=T_n(n)$ and
equality is included in case~(ii). We may therefore assume 
$r\le n-1$. Every part is independent, so
$B_{ii}=0$ for all $i$ and hence $\tau=\tr B=0$. Now
\Cref{thm:partition-closed} gives
$$K(G)\ \ge\ \frac{(r-1)^2}{r}+\frac{(n-r)^2}{n-r}=n-2+\frac1r,$$
and, since $\rho=-\tau/(n-r)=0$ and $\frac{\tau-1}{r-1}=-\frac{1}{r-1}$, equality
holds if and only if
\begin{equation}\label{eq:spec-equality}
\spec (M)=\Big\{1,\ \big(-\tfrac{1}{r-1}\big)^{(r-1)},\ 0^{(n-r)}\Big\}.
\end{equation}
For the equality classification, assume now that $G$ is unweighted and that equality holds. Then \eqref{eq:spec-equality} holds, so
$M$ has exactly one positive eigenvalue,
and since $A=D^{1/2}MD^{1/2}$ with $D^{1/2}$ invertible, the matrices $A$ and $M$
are congruent, so by Sylvester's law of inertia $A$ has exactly one positive
eigenvalue as well. By Smith's theorem \cite{Smith1970}, $G$ is complete
multipartite. Its number of parts is $\chi(G)=r$, and as $\mathcal P$ consists of
$r$ independent sets it must be the partition into those parts; write
$G\cong K_{n_1,\dots,n_r}$ with $V_i$ the part of size $n_i$.

It remains to show that the parts are balanced. Equality in
\Cref{thm:partition-closed} implies equality in
\Cref{thm:partition-pinching}, so
$\spec (B)\uplus\{0^{(n-r)}\}=\spec (M)$; cancelling
$\{0^{(n-r)}\}$ in \eqref{eq:spec-equality} gives
$\spec (B)=\{1,(-\tfrac{1}{r-1})^{(r-1)}\}$.
By the proof
of \Cref{thm:partition-pinching}, $1$ is the largest eigenvalue of $B$, with unit
positive eigenvector $\mathsf c$.
Then $B+\tfrac{1}{r-1}\II$ is positive semidefinite of rank one with trace
$\tfrac{r}{r-1}$, so
$$B+\frac{1}{r-1}\II
=\frac{r}{r-1}\mathsf{c}\mathsf{c}^{\top}.$$
Reading the diagonal, $B_{ii}=0$ forces $\mathsf c_i^{2}=\frac1r$, that is
$\vol (V_i)=\vol (G)/r$ for every $i$. Since
$\vol (V_i)=n_i(n-n_i)$, all the $n_i$ satisfy one quadratic, so
$n_i\in\{a,n-a\}$ for some $1\le a\le n/2$. If $r\ge3$ and some part had size
$n-a$, the remaining $r-1\ge2$ parts would have sizes summing to $a$ while each
is at least $a$, which is impossible; hence $n_1=\cdots=n_r=n/r$, so $r\mid n$
and $G\cong T_r(n)$. For $r=2$ the condition $n_1n_2=n_2n_1$ is vacuous and $G$
is an arbitrary complete bipartite graph.

Conversely, $\spec (M(K_{p,q}))=\{1,0^{(n-2)},-1\}$ gives
$K=n-\frac32$, and
$\spec (M(T_r(n)))=\{1,(-\tfrac{1}{r-1})^{(r-1)},0^{(n-r)}\}$ gives
$K=\tfrac{(r-1)^2}{r}+(n-r)=n-2+\tfrac1r$; both are instances of
\eqref{eq:spec-equality}.
\end{proof}

\subsubsection{Comparison with the normalized Hoffman bound}

For the remainder of this subsection, $G$ is unweighted.
By \Cref{cor:kemeny-chromatic}, Kemeny's constant gives the following lower bound for the chromatic number of a graph

\begin{corollary}[Kemeny's bound on the chromatic number]
$$\chi(G)\ge \frac{1}{K(G)-n+2}\eqqcolon\beta_K(G).$$
\end{corollary}
    
	Since $\chi(G)$ is an integer, the corresponding integer-valued bound is
	$$\chi(G)\ge\left\lceil\frac{1}{K(G)-n+2}\right\rceil.$$	
	Next we compare our with the normalized-Laplacian version of Hoffman's bound, which was shown by 
	 Chung \cite{Chung1997} using the second largest eigenvalue of the normalized Laplacian, $\mu_1$.

\begin{theorem}[Normalized Hoffman bound on the chromatic number {\cite[Theorem~6.7]{Chung1997}}]
Let \(G\) be a graph with no isolated vertices. Let
$2 \geq \mu_1 \geq \cdots \geq \mu_n = 0$ be the eigenvalues of its normalized Laplacian, and let
$1=\lambda_1 \geq \cdots \geq \lambda_n \geq -1$ be the eigenvalues of its normalized adjacency matrix. Then
\[
    \chi(G)
    \geq
    \frac{\mu_1}{\mu_1-1}
    =
    1-\frac{1}{\lambda_n}.
\]
\end{theorem}

	Following Elphick and Wocjan notation
	\cite{ElphickWocjan2015}, we refer to this inequality as the
	\emph{normalized Hoffman bound}. 
	In general, neither $\beta_K(G)$ nor $\beta_H(G)$ dominates the other.
	
	We first give a family for which our Kemeny bound is stronger. For
	$n\ge4$ and $2\le t\le n-2$, let
	$$	G_{n,t}=\overline{K_t}\vee K_{n-t}
	=K_{t,1,\ldots,1}.$$
	Thus, $G_{n,t}$ is obtained from $K_n$ by deleting all the edges of a
	fixed copy of $K_t$, and
	$\chi(G_{n,t})=n-t+1$. The vertices in $\overline{K_t}$ have degree
	$n-t$, while those in $K_{n-t}$ have degree $n-1$.
	
	The vectors supported on $\overline{K_t}$ whose coordinates sum to zero
	give the normalized adjacency eigenvalue $0$ with multiplicity $t-1$.
	Similarly, the vectors supported on $K_{n-t}$ whose coordinates sum to
	zero give the eigenvalue $-1/(n-1)$ with multiplicity $n-t-1$. On the
	subspace of vectors that are constant on each of the two parts, the
	normalized adjacency matrix has compression
	$$
	\begin{pmatrix}
		0&\sqrt{\dfrac{t}{n-1}}\\[2mm]
		\sqrt{\dfrac{t}{n-1}}&
		\dfrac{n-t-1}{n-1}
	\end{pmatrix},
	$$
	whose eigenvalues are $1$ and $-t/(n-1)$. Consequently,
	$$	\spec(M(G_{n,t}))
	=\left\{
	1,-\frac{t}{n-1},
	\left(-\frac{1}{n-1}\right)^{(n-t-1)},
	0^{(t-1)}
	\right\}.	$$
	It follows that
	$$	\beta_K(G_{n,t})=	\frac{n(n+t-1)}{n+t^2-1}.	$$
	On the other hand, since $\lambda_n=-t/(n-1)$, the normalized Hoffman
	bound gives
	$$	\beta_H(G_{n,t})	=	\frac{n+t-1}{t}.	$$
	A direct calculation shows that
	\begin{equation}\label{eq:bet_K-beta_H}
    \beta_K(G_{n,t})-\beta_H(G_{n,t})	=
	\frac{(t-1)(n-t-1)(n+t-1)}	{t(n+t^2-1)}>0.
    \end{equation}
	Therefore, the Kemeny bound is strictly stronger than the normalized
	Hoffman bound on the entire family $G_{n,t}$, for
	$2\le t\le n-2$. The difference \eqref{eq:bet_K-beta_H} even tends to $\infty$ as $n$ grows provided that $t=o(n)$.	
	In particular, when $t=2$,
	$$\left\lceil\beta_K(G_{n,2})\right\rceil =\left\lceil\frac{n(n+1)}{n+3}\right\rceil =n-1	=\chi(G_{n,2}),	\quad
	\text{whereas}\quad\beta_H(G_{n,2})=\frac{n+1}{2}.$$
	
For the opposite comparison, consider the generalized friendship graph
$$F_{p,q}=K_1\vee(pK_q),\qquad p,q\ge2.$$
Thus, $F_{p,q}$ is obtained by joining one vertex to all vertices of
$p$ disjoint copies of $K_q$. It has order $n=pq+1$ and chromatic
number $\chi(F_{p,q})=q+1.$

The  normalized adjacency spectrum of $F_{p,q}$ is
\cite{BermanEtAl2018}
$$\spec(M(F_{p,q}))
=\left\{1,\left(\frac{q-1}{q}\right)^{(p-1)},
\left(-\frac1q\right)^{(pq-p+1)}
\right\}.$$
This spectrum can also be verified directly. Vectors whose coordinates
sum to zero within one of the copies of $K_q$ give the eigenvalue
$-1/q$, differences between vectors that are constant on the copies of
$K_q$ give the eigenvalue $(q-1)/q$, and the remaining two-dimensional
subspace gives the eigenvalues $1$ and $-1/q$.

Since the largest normalized Laplacian eigenvalue is
$\mu_1=1+1/q$, the normalized Hoffman bound gives
$$\beta_H(F_{p,q})=\frac{\mu_1}{\mu_1-1}=q+1=\chi(F_{p,q}).$$
On the other hand,
$$K(F_{p,q})=q(p-1)+\frac{q}{q+1}(pq-p+1).$$
Since $n=pq+1$, it follows that
$$\beta_K(F_{p,q})=\frac{q+1}{1+(p-1)q(q-1)}.$$
For $p,q\ge2$,
$$\beta_K(F_{p,q})\le1<q+1=\beta_H(F_{p,q}),$$
where equality in the first inequality occurs only when $p=q=2$.
Consequently, on generalized friendship graphs, the normalized Hoffman
bound determines the chromatic number exactly, whereas the Kemeny bound
is trivial.

These two families show that the Kemeny bound and the normalized Hoffman
bound are genuinely incomparable. They nevertheless coincide, and are
both exact, for balanced complete multipartite graphs.

\subsection{The normalized cut}\label{sec:ncut}

We next consider the normalized cut of a vertex partition. In this setting,
$\tau(\mathcal P)$ is directly related to the normalized cut, which is a
standard objective function in spectral clustering. Thus,
\Cref{thm:partition-closed} gives a bound on $K(G)$ in terms of the normalized
cut of the partition.

Let $\mathcal P=\{V_1,\dots,V_m\}$ be a partition of $V(G)$ into nonempty parts
and let
$$\ncut(\mathcal P) \coloneqq\ \sum_{i=1}^{m}\frac{\cut(V_i,\bar V_i)}{\vol(V_i)}$$
be its \emph{normalized cut}. 
The normalized cut is the objective introduced by Shi and Malik
\cite{ShiMalik2000} for image segmentation and is the standard multiway
partition functional of spectral clustering, see von Luxburg
\cite{vonLuxburg2007} for a survey. The quantity $\operatorname{cut}(V_i,\bar V_i)/\operatorname{vol}(V_i)$ is the
probability that the stationary random walk on $G$, conditioned to be in $V_i$,
leaves $V_i$ in one step, so $\operatorname{NCut}(\mathcal P)$ is the total
one-step escape rate of the partition. For a random-walk parameter such as
$K(G)$ this, rather than the cut normalized by cardinality, is the natural
quantity.

The bounds we obtain in this and the next subsection are defined in terms of the following function:
\[F_m(x)\coloneqq\frac{(m-1)^2}{x}+\frac{(n-m)^2}{n-x},\qquad 2\le m\le n-1,\quad 0<x<n.\]

\begin{proposition}\label{prop:ncut}
 We have
\begin{equation}\label{eq:F-form}
K(G) \ge F_m\big(\ncut(\mathcal P)\big).
\end{equation}
\end{proposition}
\begin{proof}
Let ${\sf S}=({\sf s}_1\ \cdots\ {\sf s}_m)$.
We first verify the identity $\ncut(\mathcal P)=m-\tau(\mathcal P)$.
Since $\mathcal L_G=I-M$ and $\|{\sf s}_i\|=1$,
$${\sf s}_i^{\top}\mathcal L_G{\sf s}_i=1-B_{ii}
=1-\frac{2e(V_i)}{\vol (V_i)}
=\frac{\cut (V_i,\bar V_i)}{\vol (V_i)},$$
using $\vol (V_i)=2e(V_i)+\cut (V_i,\bar V_i)$.
Summing over $i$ gives
$\tr ({\sf S}^{\top}\mathcal L_G{\sf S})=m-\tr B
=m-\tau$. Each summand lies in $(0,1]$,
because $G$ is connected and each part is nonempty and proper. Hence $0<\ncut(\mathcal P)\le m<n$.
Substituting $m-\tau=\ncut $ and
$n-m+\tau=n-\ncut $ into \eqref{eq:partition-closed} gives
\eqref{eq:F-form}.
\end{proof}

\subsection{Conductance}\label{sec:conductance-cor}
Finally, we apply the same bound to partitions whose parts have small
conductance. This gives a lower bound on $K(G)$ in terms of the conductance
of the partition, and hence in terms of the conductance of the graph.

 Recall that the \emph{conductance} of a graph $G$ is defined as
$$\phi(G) = \min_{\emptyset\ne S \subsetneq V} \frac{\sum_{u \in S, v \in \bar{S}} c(u,v)}
{\min\bigl({\sf vol}(S),{\sf vol}(\bar{S})\bigr)},$$
and recall the well-known  \emph{Cheeger's inequality} on $\phi(G)$ \cite{Chung1997}:
$$\frac{\phi(G)^2}{2} \leq \mu_{n-1} \leq 2\phi(G).$$

\begin{lemma}\label{prop:two-regimes}
Let $2\le m\le n-1$. Then $F_m$ is strictly convex on $(0,n)$, strictly
decreasing on $(0,x^{*}]$ and strictly increasing on $[x^{*},m]$, where
$$x^{*}=\frac{n(m-1)}{n-1},\qquad F_m(x^{*})=\frac{(n-1)^2}{n}.$$
Moreover $F_m(m)=n-2+\frac1m$.
\end{lemma}

\begin{proof}
$F_m'(x)=-\frac{(m-1)^2}{x^2}+\frac{(n-m)^2}{(n-x)^2}$ vanishes exactly when
$x(n-m)=(m-1)(n-x)$, i.e.\ $x=x^{*}$, and $F_m''>0$ on $(0,n)$. Substituting
$x^{*}$ and $n-x^{*}=\frac{n(n-m)}{n-1}$ gives
$$F_m(x^{*})=\frac{(n-1)^2}{n},\quad \text{and}\quad
F_m(m)=n-2+\frac1m.$$ Finally $x^{*}<m$ because
$n(m-1)<m(n-1)$ for $m<n$.
\end{proof}

\begin{corollary}\label{cor:clustering}
Let $G$ be a connected graph of order $n\ge3$, let $2\le m\le n-1$, and suppose $V(G)$ admits a partition into $m$ nonempty parts such that
$\cut (V_i,\bar V_i)\le\phi\vol(V_i)$ for $i=1,\dots,m$. Then
$$K(G)\ \ge\ \frac{(m-1)^2}{m\phi}+\frac{(n-m)^2}{n}.$$
\end{corollary}

\begin{proof}
Put $x=\ncut(\mathcal P)$. Then $0<x\le m\phi$, and \Cref{prop:ncut} gives
$$K(G)\ge\frac{(m-1)^2}{x}+\frac{(n-m)^2}{n-x}
\ge\frac{(m-1)^2}{m\phi}+\frac{(n-m)^2}{n-x}
>\frac{(m-1)^2}{m\phi}+\frac{(n-m)^2}{n}.$$
\end{proof}

\begin{corollary}\label{cor:conductance}
Let $G$ be a connected graph of order $n\ge3$. Then
$$K(G) \ge \frac{1}{2\phi(G)}+\frac{(n-2)^2}{n}.$$
In fact, equality cannot occur.
\end{corollary}
\begin{proof}
Choose a nonempty proper set $S$ attaining $\phi(G)$ and put $b=\cut(S,\bar S)$. Both $b/\vol(S)$ and $b/\vol(\bar S)$ are at most $\phi(G)$. Applying \Cref{cor:clustering} with $m=2$ gives the stated bound, and the strict inequality in its proof shows that equality cannot occur.
\end{proof}

The bound is asymptotically sharp, as illustrated by the following example.
\begin{example}
Let $\mathcal B_s$ consist of two disjoint copies of $K_s$, joined by one edge between designated vertices, where $s\ge2$. Set $v_s=s(s-1)+1$. Each clique has volume $v_s$, and the cut consisting of the joining edge has conductance $1/v_s$. Every nontrivial cut has at least one edge and smaller-side volume at most $v_s$, so
$$\phi(\mathcal B_s)=\frac{1}{v_s}.$$
Exact formulas for Kemeny's constant of barbell graphs were obtained in \cite{Breen2019KemenyBarbell}. A simpler formula for the family $B(1,a,b,c)$ was later given in \cite[Theorem~3.4]{FKK2022}. In the notation of \cite{FKK2022}, $\mathcal B_s=B(1,2,s-1,s-1)$, and hence their formula gives
\begin{equation}\label{eq:barbell-exact}
K(\mathcal B_s)=\frac{2(s-2)(s-1)}s+\frac{s(s-1)}{v_s}+\frac{s^2+s-1}{2}.
\end{equation}
The lower bound in \Cref{cor:conductance} is
$$L_s=\frac{v_s}{2}+\frac{(2s-2)^2}{2s},$$
and subtraction gives the exact positive gap
$$K(\mathcal B_s)-L_s=\frac{(s-1)^2(s^2-s+2)}{s(s^2-s+1)}.$$
Both $K(\mathcal B_s)$ and $L_s$ are asymptotic to $s^2/2$, so $L_s/K(\mathcal B_s)\to1$. In particular the coefficient $1/2$ of $1/\phi(G)$ cannot be increased uniformly, even for unweighted graphs.
\end{example}

\section{Interlacing}\label{sec:interlacing}

In this section we consider unweighted graphs,  and we focus on investigating how structural properties - particularly vertex and edge removal and graph clustering - affect Kemeny's constant. Several studies have explored the relationship between graph structure and Kemeny's constant, see e.g. \cite{ABCMP2023,BCK2022,KLMZ2024,FKK2022}. 

In order to derive our main results we will use another fundamental technique from spectral graph theory: eigenvalue interlacing. We will be using both vertex and edge interlacing techniques.

Consider two sequences of real numbers: $\lambda_1 \geq\cdots \geq\lambda_{n}$ and $\theta_1\geq \cdots \geq \theta_{m}$ with $m<n$. The second sequence is said to \emph{interlace} the first one whenever
\begin{equation*}
\lambda_i\geq \theta_i \geq \lambda_{n-m+i} \quad \text{for }i=1,\ldots,m.
\end{equation*}

If $m=n-1$, the interlacing inequalities become $\lambda_1\geq \theta_1 \geq \lambda_2 \geq \theta_2\geq \cdots \geq \theta_m \geq \lambda_n$, which clarifies the term ``interlacing". Throughout, the $\lambda_i$ and the $\theta_i$ represent the eigenvalues of matrices, $A$ and $B$, respectively.

\begin{theorem}[Interlacing Theorem, see \cite{H1995}]\label{thm:interlacing}
Let $A$ be a real symmetric $n\times n$  matrix with
eigenvalues $\lambda_1\ge\cdots\ge \lambda_n$. For some $m<n$,
let $S$ be a real $n\times m$ matrix with orthonormal columns,
$S^{\top}S=I$, and consider the matrix $B=S^{\top}AS$,
with eigenvalues $\theta_1\ge\cdots\ge \theta_m$. Then, the eigenvalues of $B$ interlace those of $A$, that is, 
\begin{equation}
\label{ineq:interlacing}
\lambda_i\ge \theta_i\ge \lambda_{n-m+i},\qquad i=1,\ldots, m.
\end{equation}
\end{theorem}

Next, we  derive a bound on Kemeny's constant using the fact that the normalized adjacency matrix is  symmetric, allowing the application of interlacing, and the established connection between $K(G)$ and the eigenvalues of this matrix (see Eq. \eqref{Lovasz}).

Two interesting particular cases of Theorem \ref{thm:interlacing} are obtained by choosing appropriately the matrix $S$.

The first one recovers the well-known Cauchy interlacing.

\begin{corollary}[Cauchy interlacing]\label{cor:2.2Haemerspaper}
If $B$ is a principal submatrix of a symmetric matrix $A$, then the eigenvalues of $B$ interlace the eigenvalues of $A$.
\end{corollary}

The second case consists  in considering $\mathcal{P}=\{U_{1},\ldots,U_{m}\}$  a partition of the
vertex set $V$, with each $U_{i}\neq \emptyset$.
Let $A$ be partitioned according to $\mathcal{P}$, that is
$$A=\left[ \begin{array}{ccc}
A_{1,1} & \cdots & A_{1,m} \\
\vdots &  & \vdots \\
A_{m,1} & \cdots & A_{m,m}
 \end{array} \right],$$
where $A_{i,j}$ denotes the submatrix (block) of $A$ formed by
rows in $U_{i}$ and columns in $U_{j}$. Then, the \emph{quotient matrix} of $A$ with respect to $\mathcal{P}$ is the $m\times m$ matrix  whose entries are the average row sums of the blocks of $A$, more precisely:
\begin{equation*}
(B)_{i,j}=\frac{1}{|U_{i}|}\1_{U_i}^{\top}A_{i,j}\1_{U_j}.
\end{equation*}

\begin{corollary}[Quotient matrix interlacing \cite{H1995}]\label{cor:2.3Haemerspaper}
Suppose $B$ is the quotient matrix of a symmetric partitioned matrix $A$. Then, the eigenvalues of $B$ interlace the eigenvalues of $A$.
\end{corollary}

The most general version of edge interlacing was proposed by Hall, Pate and Stewart \cite{edgeinterlacing}.

\begin{lemma}[{\cite[Proposition 3.2]{edgeinterlacing}}]\label{thm:removingredges}
Let $G$ be a connected unweighted graph and let $H$ be a connected spanning subgraph of $G$ obtained by deleting $r$ edges.
If
$\mu_1\geq \cdots \geq \mu_n=0$
and
$\theta_1\geq \cdots \geq \theta_n=0$
are the eigenvalues of the normalized Laplacian matrices ${\cal{L}}_G$ and ${\cal{L}}_H$, respectively, then
\[\mu_{i-r} \geq \theta_i \geq \mu_{i+r} \qquad \text{for } i=1,2,\ldots, n, \vspace{-5pt}\]
with the convention of $\mu_i=2$ for each  $i\leq 0$ and $\mu_{i}=0$ for each $i\geq n+1$.
\end{lemma}

We now return to weighted graphs for the principal-submatrix and quotient-partition bounds.

Next, we apply the above results to obtain bounds on Kemeny's constant.

\begin{theorem}\label{thm:vertexinterlacingbounds}
Let $G$ be a connected graph, $D^{-1/2}AD^{-1/2}$ its normalized adjacency matrix  and $B$ a $m\times m$ matrix with eigenvalues $\theta_1\ge\cdots\ge \theta_m$, such that $B$ is
\begin{description}
    \item[$(i)$] a principal submatrix of $D^{-1/2}AD^{-1/2}$, or,
    \item[$(ii)$] the quotient matrix resulting of a partition $\{V_1,\dots,V_m\}$ of the vertex set of $G$.
\end{description}
Then,
$$
 \sum_{i=2}^{m}\frac{1}{1-\theta_i} + \dfrac{(n-m)}{2} \leq K(G)\leq \sum_{i=2}^{m}\frac{1}{1-\theta_i} + \dfrac{(n-m)}{1-\lambda_2}.
$$ 
\end{theorem}

\begin{proof} (i) By Corollary \ref{cor:2.2Haemerspaper}, it follows that $\lambda_i\geq \theta_i$  and in consequence $\frac{1}{1-\lambda_i} \geq \frac{1}{1-\theta_i}  \text{ for }i=2,\ldots,m$. Moreover, since for the normalized adjacency matrix it holds that $\lambda_i\geq -1$ \cite{Lo16}, it follows that $\frac{1}{1-\lambda_i} \geq \frac{1}{2} \text{ for }i=m+1,\ldots,n$. From Eq. \eqref{Lovasz} we obtain the desired lower bound. 

To get the upper bound we observe that from Corollary \ref{cor:2.2Haemerspaper} 
$$\dfrac{1}{1-\lambda_{n-m+j}}\le \dfrac{1}{1-\theta_j}, \hspace{.25cm}j=2,\ldots,m.$$
Moreover, $\frac{1}{1-\lambda_i} \leq \frac{1}{1-\lambda_2}$ and hence
$$K(G)=\sum_{i=2}^{n}\frac{1}{1-\lambda_i}\leq \sum_{i=2}^{m}\frac{1}{1-\theta_i} + \dfrac{(n-m)}{1-\lambda_2}.$$

(ii) Analogously, the result follows  using Eq. \eqref{Lovasz} and applying Corollary \ref{cor:2.3Haemerspaper}. 
\end{proof} 

Similarly as we did in previous sections, Theorem \ref{thm:vertexinterlacingbounds} can also be used to relate Kemeny's constant to graph
parameters that are naturally described by vertex partitions or by induced substructures (we will illustrate it in Corollary \ref{coro:interlacinggraphconductance}). Indeed, whenever a graph parameter determines a partition of the vertex set, the associated quotient matrix may be
inserted into Theorem \ref{thm:vertexinterlacingbounds}(ii), while parameters characterized through a distinguished vertex subset may lead to a principal submatrix to
which Theorem \ref{thm:vertexinterlacingbounds}(i) applies. In this way, spectral information coming from such structural descriptions can be converted directly into bounds on \(K(G)\). This is potentially useful in particular for graph
parameters whose exact computation is difficult (NP-hard), since any admissible
partition or substructure already yields a computable spectral bound.

For the lower bound from Theorem \ref{thm:vertexinterlacingbounds}$(i)$ to be close to equality, the omitted terms $1/(1-\lambda_i)$ must be close to $1/2$, so the corresponding normalized-adjacency eigenvalues must be close to $-1$.

\begin{corollary}\label{coro:vertexinterlacingbounds}
  Let $G$ be a connected unweighted graph with degrees $k_1,\dots,k_n$ and normalized
  adjacency eigenvalues $1=\lambda_1>\lambda_2\ge\cdots\ge\lambda_n$.
  Then
\begin{equation}\label{eq:K(G)<}
    K(G)\ \le\ \min_{i\sim j}\frac{\sqrt{k_ik_j}}{\sqrt{k_ik_j}+1}
              +\frac{n-2}{1-\lambda_2}.
\end{equation}
	Equality holds if and only if
$G\cong K_n$.
\end{corollary}

\begin{proof}
  Apply Theorem~\ref{thm:vertexinterlacingbounds}$(i)$ with $m=2$ and $S$
  the two coordinate vectors of an adjacent pair $i\sim j$. Then $B$ is
  the $2\times2$ block 
  $$B^{(ij)}=
  \begin{pmatrix}0&\frac1{\sqrt{k_ik_j}}\\\frac1{\sqrt{k_ik_j}}&0\end{pmatrix},$$
  so $\theta_2=-1/\sqrt{k_ik_j}$ and
  $\tfrac{1}{1-\theta_2}=\tfrac{\sqrt{k_ik_j}}{\sqrt{k_ik_j}+1}<1$.
  The upper bound of Theorem~\ref{thm:vertexinterlacingbounds} holds for
  every such pair; a non-adjacent pair would instead give
  $\tfrac{1}{1-\theta_2}=1$, which is never smaller. Minimizing over
  adjacent pairs gives the claim.

Now, we characterize the equality case. First suppose that $G=K_n$. Then $k_i=n-1$ for every vertex $i$, and
  $$\lambda_2=\cdots=\lambda_n=-\frac1{n-1}. $$
  Hence both sides of \eqref{eq:K(G)<} have the equal value
  $\frac{(n-1)^2}{n}$.

  Conversely, if $n=2$, then the only connected graph is $K_2$, and the result is immediate. We therefore assume that $n\ge 3$.
     Let  $uv$ be an edge with 
  $a=1/\sqrt{k_uk_v}=\max_{i\sim j} 1/\sqrt{k_i k_j}$, and assume that the euqity holds for $uv$, that is
   $$  K(G)=\frac{\sqrt{k_uk_v}}{\sqrt{k_uk_v}+1}
  +\frac{n-2}{1-\lambda_2}.  $$
  
    The eigenvalues of $B^{(uv)}$ are
  $ \theta_1=a$ and $ \theta_2=-a$.
  Consequently,
  $$
  \frac1{1-\theta_2}
  =
  \frac1{1+a}
  =
  \frac{\sqrt{k_uk_v}}{\sqrt{k_uk_v}+1}.
  $$
  
  The upper bound corresponding to the edge $uv$ is obtained as follows:
  \begin{equation}\label{eq:K(G)=for-uv}
  K(G)=\sum_{i=2}^{n-1}\frac1{1-\lambda_i}
  +\frac1{1-\lambda_n}\le
  \frac{n-2}{1-\lambda_2}
  +
  \frac1{1-\theta_2}=
  \frac{n-2}{1-\lambda_2}
  +  	\frac1{1+a}. 
  \end{equation} 
  Indeed, $\lambda_i\le\lambda_2$ for $i=2,\ldots,n-1$, while Cauchy interlacing gives
  $\lambda_n\le\theta_2=-a.$
  Equality in \eqref{eq:K(G)=for-uv} implies
  $ \lambda_2=\lambda_3=\cdots=\lambda_{n-1}=:c  $ and $\lambda_n=\theta_2=-a.$
  Thus the normalized adjacency spectrum of $G$ is
  $ \spec(M)
  =  \{1,c^{(n-2)},-a\}. $
  
  Let  $ \x=\frac{1}{\sqrt2}({\sf e}^{u}-{\sf e}^{v}).$
  Since the only nonzero entries of $\x$ correspond to $u$ and $v$, we have
  $ \x^\top M\x=-a=\lambda_n.$
  Since $\lambda_n$ is the smallest eigenvalue of the symmetric matrix $M$, equality in the Rayleigh quotient implies
  $ M\x=-a\x.$
  
  Let
  $$\z= \frac{1} {\sqrt{\vol(G)}}D^{1/2}\1.$$
  Then $\z$ is a positive unit eigenvector of $M$ corresponding to the eigenvalue $1$. Since $M$ is symmetric and $1\ne-a$, the eigenvectors $\z$ and $\x$ are orthogonal. Therefore
  $$
  0=\z^\top\x
  =
  \frac{\sqrt{k_u}-\sqrt{k_v}}
  {\sqrt{2\vol(G)}},
  $$
  and hence
  $ k_u=k_v.$
  
  By the spectral theorem, the eigenspace corresponding to $c$ is the orthogonal complement of the span of $\z$ and $\x$. Therefore
  $$
  M
  =
  cI+(1-c)\z\z^\top+(-a-c)\x\x^\top.
  $$
  This identity remains valid when $c=-a$, because in that case the last coefficient is zero.
  
  Now let $s,t$ be two distinct vertices such that $\{s,t\}\ne\{u,v\}$. Since $\x$ is supported only on $u$ and $v$, we have
  $ x_sx_t=0.$
  Hence
  $  M_{st}=(1-c)z_sz_t.$
  The graph $G$ is connected, so the eigenvalue $1$ of $M$ is simple, and therefore $c<1$. Also, every entry of $\z$ is positive. It follows that
  $  M_{st}>0.$
  Since
  $ M_{st}=\frac{A_{st}}{\sqrt{k_sk_t}},$
  we obtain $A_{st}=1$. Thus every two distinct vertices other than possibly $u$ and $v$ are adjacent. Since $uv\in E(G)$ by construction, every two distinct vertices of $G$ are adjacent. Therefore
  $ G\cong K_n.$  
\end{proof}

It is well known that adding an edge to a graph can cause Kemeny's constant to increase, decrease or remain unchanged. Indeed, consider for example the graphs minimizing (maximizing) Kemeny's constant, the complete graphs (barbell graphs), and remove an edge -- this immediately increases (decreases) the constant. Kirkland et al. \cite{KLMZ2024} provided a quantitative analysis of this behavior when the initial graph is a tree of fixed order. In this section, we  analyze how Kemeny's constant changes when multiple edges are removed.

To do so, we apply the edge interlacing technique introduced by Hall, Patel, and Stewart \cite{edgeinterlacing} (see Lemma \ref{thm:removingredges}). As a particular case, we show that the removal of a single branch follows directly as a corollary of our more general results.

\begin{theorem}\label{thm:detemingredges}
Let $G$ be a connected unweighted graph, and let $H$ be a connected spanning subgraph of $G$ obtained by deleting $r\leq n-1$ edges. If
$\mu_1\geq \cdots \geq \mu_n=0$
and
$\theta_1\geq \cdots \geq \theta_n=0$
are the eigenvalues of the normalized Laplacian matrices ${\cal{L}}_G$ and ${\cal{L}}_H$, respectively. Then

$$\frac{r}{\mu_1}+\sum_{j=1}^{n-r-1}\frac{1}{\theta_{j}} \leq K(G)\leq \sum_{j=r+1}^{n-1} \frac{1}{\theta_{j}}+\frac{r}{\mu_{n-1}}.$$

\end{theorem}
\begin{proof}
For $1\le j\le n-r-1$, Lemma~\ref{thm:removingredges} gives
$$\mu_j\ge\theta_{j+r},\qquad \theta_j\ge\mu_{j+r}.$$
All eigenvalues in these inequalities are positive. Taking reciprocals and splitting the sums gives
\begin{align*}
K(G)&=\sum_{i=1}^{r}\frac{1}{\mu_i}+\sum_{j=1}^{n-r-1}\frac{1}{\mu_{j+r}}
\ge\frac{r}{\mu_1}+\sum_{j=1}^{n-r-1}\frac{1}{\theta_j},\\
K(G)&=\sum_{j=1}^{n-r-1}\frac{1}{\mu_j}+\sum_{i=n-r}^{n-1}\frac{1}{\mu_i}
\le\sum_{j=1}^{n-r-1}\frac{1}{\theta_{j+r}}+\frac{r}{\mu_{n-1}}.
\end{align*}
Reindexing the first sum in the upper bound gives the stated result. Empty sums cover the endpoint cases.
\end{proof}

%
%

Observe that the bounds in the above theorem can be written as
$$K(H)-\sum_{j=n-r}^{n-1}\frac{1}{\theta_{j}}+\frac{r}{\mu_1} \leq K(G)\leq K(H)-\sum_{j=1}^{r} \frac{1}{\theta_{j}}+\frac{r}{\mu_{n-1}}.$$


\begin{corollary}\label{coro:removingoneedge}
Let $G$ be a connected unweighted graph, and let $H$ be a subgraph of $G$ obtained by deleting an edge such that $H$ is connected. Then
$$K(H)-\frac{1}{\theta_{n-1}}+\frac{1}{\mu_1} \leq K(G) \leq K(H)-\frac{1}{\theta_1}+ \frac{1}{\mu_{n-1}}.$$  
\end{corollary}

\begin{example}
Let $n\ge4$ and consider the complete graph $K_n$, we know that $\mu_n=0$ and $\mu_i=n/(n-1)$, $i=1,\ldots,n-1$ and $K(K_n)=(n-1)^2/n.$ Let $H=K_n\setminus\{e\}$, where $e$ is an edge of $K_n.$ In this case, we can compute the eigenvalues for the normalized Laplacian and we obtain
$$\left\{ \frac{n+1}{n-1},\,  \left(\frac{n}{n-1}\right)^{(n-3)},\,  1,\,  0 \right\}.$$
Then, $$K(H)=1+\dfrac{n-1}{n+1}+\dfrac{(n-1)(n-3)}{n}.$$
Applying Corollary \ref{coro:removingoneedge}, we obtain the following bounds 
$$\dfrac{(n-1)(n-2)}{n}+\dfrac{(n-1)}{(n+1)}\le \dfrac{(n-1)^2}{n}\le \dfrac{(n-1)(n-2)}{n}+1,$$
which are asymptotically tight. 
\end{example}

  As a consequence of Corollary \ref{coro:removingoneedge} we obtain the following result, involving the graph conductance. 
	\begin{corollary}\label{coro:interlacinggraphconductance}
		Let $G$ be a connected unweighted graph and let $e$ be an edge of $G$ such that $H = G - \left\{e\right\}$ is connected. Then $$K(H)\geq K(G)+\frac{1}{2}-\frac{2}{\phi(G)^2}.$$
	\end{corollary}

\subsection*{Declaration of AI}
The authors acknowledge the use of ChatGPT (GPT-5.6; accessed September 2026) to explore proof strategies, and subsequently for language editing and symbolic and numerical checks. All formal statements, arguments, and proofs in the manuscript were written and verified by the authors, who take full responsibility for the accuracy and integrity of the article.

\subsection*{Acknowledgements}
 
Aida Abiad is supported by the Dutch Research Council (NWO) through the grant \linebreak VI.Vidi.213.085. This work has been partially supported by the Spanish
Research Council under project 
PID2021-122501NB-I00 and by the Universitat Polit\`ecnica de Catalunya under funds AGRUPS-UPC 2025. \'A. Samperio was supported by a FPI grant of the Research Project PGC2018-096446-BC21. Part of this work was done  while the first author was visiting the Simons Institute for the Theory of Computing at UC Berkeley.

 


\printbibliography[heading=bibintoc]

\end{document}